\documentclass[a4paper,12pt]{amsart}
\usepackage{amssymb,amscd}
\usepackage[cp1251]{inputenc}

\usepackage{amsmath}
\usepackage{amsthm}
\usepackage{mathtext}
\usepackage{euscript}

\usepackage{graphicx}
\graphicspath{}
\DeclareGraphicsExtensions{.pdf,.png,.jpg,.jpeg}

\linethickness{0.5pt}

\usepackage[english]{babel}

\theoremstyle{plain}

\newtheorem{theo}{Theorem}[section]
\newtheorem{prop}[theo]{Proposition}

\newtheorem{coro}[theo]{Corollary}

\theoremstyle{definition}
\newtheorem*{defi}{Definition}
\newtheorem{exam}[theo]{Example}

\theoremstyle{remark}
\newtheorem*{rema}{Remark}

\numberwithin{equation}{section}
\numberwithin{figure}{section}

\newcommand{\mb}[1]{{\textbf {\textit#1}}}

\newcommand{\field}[1]{\mathbb{#1}}
\def\C{\field{C}}

\newcommand{\R}{\field{R}}

\DeclareMathOperator{\cc}{cc}
\DeclareMathOperator{\Tor}{Tor}

\DeclareMathOperator{\rk}{rk}

\DeclareMathOperator{\conv}{conv}

\def\le{\leqslant}
\def\ge{\geqslant}

\newcommand{\zp}{\mathcal Z_P}
\newcommand{\zk}{\mathcal Z_K}
\newcommand{\ko}{\Bbbk}

\begin{document}

\title[Moment-angle manifolds]{Moment-angle manifolds, 2-truncated cubes and Massey operations}

\author{Ivan Limonchenko}

\thanks{The author was supported by the General Financial Grant from the China Postdoctoral Science Foundation, Grant No. 2016M601486.
}


\address{School of Mathematical Sciences, Fudan University, 220 Handan Road, Shanghai, 200433, P.R. China}
\email{ilimonchenko@gmail.com}


\begin{abstract}

We construct a family of manifolds, one for each $n\geq 2$, having a nontrivial Massey $n$-product in their cohomology. These manifolds turn out to be smooth closed 2-connected manifolds with a compact torus $\mathbb T^m$-action called moment-angle manifolds $\mathcal Z_P$, whose orbit spaces are simple $n$-dimensional polytopes $P$ obtained from a $n$-cube by a sequence of truncations of faces of codimension 2 only (\textit{2-truncated cubes}). Moreover, the polytopes $P$ are flag nestohedra but not graph-associahedra. We compute some bigraded Betti numbers $\beta^{-i,2(i+1)}(Q)$ for an associahedron $Q$ in terms of its graph structure and relate it to the structure of the loop homology (Pontryagin algebra) $H_{*}(\Omega\mathcal Z_Q)$. We also study triple Massey products in $H^{*}(\mathcal Z_Q)$ for a graph-associahedron $Q$.

\end{abstract}

\maketitle

\section{Introduction}
The main aim of this work is to show that one of the key objects of study in toric topology -- the moment-angle manifold $\zp$ of a simple convex $n$-dimensional polytope $P$ -- gives us an example of a smooth closed 2-connected manifold with a compact torus action such that its rational cohomology ring may contain a nontrivial higher Massey product of order $n$. The corresponding polytopes $P$ are 2-truncated cubes and, moreover, flag nestohedra, see~\cite{Post05},~\cite{PRW06}. The class of 2-truncated polytopes was studied in toric topology by Buchstaber and Volodin, who proved that flag nestohedra can be realized as 2-truncated cubes and that Gal conjecture on $\gamma$-vectors of simple polytopes holds for 2-truncated cubes and, therefore, for all flag nestohedra~\cite{bu-vo11}. We generalize in the polytopal sphere case the result of Baskakov~\cite{BaskM} who constructed a family of triangulated spheres $K$ whose moment-angle complexes $\mathcal Z_K$ have nontrivial triple Massey products of 3-dimensional classes in $H^{*}(\mathcal Z_K)$. In the lowest dimension Baskakov's construction gives a 2-sphere with 8 vertices $K$ -- the only $K$ with a nontrivial triple Massey product in $H^{*}(\mathcal Z_K)$ among all the fourteen 2-spheres on 8 vertices. Denham and Suciu~\cite{DS} generalized the result of Baskakov by proving a combinatorial criterion for $K$ to give a $\mathcal Z_K$ with a nontrivial triple Massey product of 3-dimensional classes in $H^{*}(\mathcal Z_K)$.

Denote by $K$ a simplicial complex of dimension $n-1$ on the vertex set $[m]=\{1,\ldots,m\}$ and by $\ko$ the base field or the ring of integers. Let $\ko[v_{1},\ldots,v_{m}]$ be a graded polynomial algebra on $m$ variables, $\deg(v_{i})=2$. The \emph{Stanley--Reisner ring} (or the \emph{face ring}) of $K$ over $\ko$ is the quotient ring
$$
   \ko[K]=\ko[v_{1},\ldots,v_{m}]/\mathcal I_K,
$$
where $\mathcal I_K$ is the ideal generated by square free
monomials $v_{i_{1}}\cdots{v_{i_{k}}}$ such that $\{i_{1},\ldots,i_{k}\}$ is not a simplex in $K$. The monomial ideal
$\mathcal I_{K}$ is called the \emph{Stanley--Reisner ideal} of~$K$.
Then $\ko[K]$ has a structure of a $\ko$-algebra and a module over $\ko[v_{1},\ldots,{v_{m}}]$ via the quotient projection. 

In what follows we denote by $P$ a simple $n$-dimensional convex polytope with $m$ {\textit{facets}} (i.e faces of codimension 1) $F_{1},\ldots,F_{m}$. 
Such a polytope $P$ can be defined as a bounded intersection of $m$ halfspaces:
$$
  P=\bigl\{\mb x\in\R^n\colon\langle\mb a_i,\mb
  x\rangle+b_i\ge 0\quad\text{for }
  i=1,\ldots,m\bigr\},\eqno (*)
$$
where $\mb a_i\in\R^n$, $b_i\in\R$. We assume that the hyperplanes
defined by the equations $\langle\mb a_i,\mb x\rangle+b_i=0$ are
in general position, that is, at most $n$ of them meet at a single
point. We also assume that there are no redundant inequalities
in $(*)$, that is, no inequality can be removed
from $(*)$ without changing~$P$. 
Then the \emph{facets} of $P$ are given by
$$
  F_i=\bigl\{\mb x\in P\colon\langle\mb a_i,\mb
  x\rangle+b_i=0\bigr\},\quad\text{for } i=1,\ldots,m.
$$

Let $A_P$ be the $m\times n$ matrix of row vectors $\mb a_i$, and
denote by $\mb b_P$ the column vector of scalars $b_i\in\R$. Then we
can rewrite $(*)$ 
as
\[
  P=\bigl\{\mb x\in\R^n\colon A_P\mb x+\mb b_P\ge\mathbf 0\}.
\]
Consider the affine map
\[
  i_P\colon \R^n\to\R^m,\quad i_P(\mb x)=A_P\mb x+\mb b_P
\]
which embeds $P$ into
\[
  \R^m_\ge=\{\mb y\in\R^m\colon y_i\ge 0\quad\text{for }
  i=1,\ldots,m\}.
\]

\begin{defi}
We define the space $\mathcal Z_P$ as a pullback in the following
commutative diagram~\cite[Lemma 3.1.6, Construction 3.1.8]{bu-pa00-2}:
$$\begin{CD}
  \mathcal Z_P @>i_Z>>\C^m\\
  @VVV\hspace{-0.2em} @VV\mu V @.\\
  P @>i_P>> \R^m_\ge
\end{CD}\eqno 
$$
where $\mu(z_1,\ldots,z_m)=(|z_1|^2,\ldots,|z_m|^2)$. The latter
map may be thought of as the quotient map for the coordinatewise
action of the standard torus
\[
  \mathbb T^m=\{\mb z\in\C^m\colon|z_i|=1\quad\text{for }i=1,\ldots,m\}
\]
on~$\C^m$. Therefore, $\mathbb T^m$ acts on $\zp$ with a quotient space $P$, $i_Z$ is a $\mathbb T^m$-equivariant embedding with a trivial normal bundle, and $\zp$ is embedded into $\C^m$ as a nondegenerate intersection of Hermitian quadrics. One can easily see that $\mathcal Z_P$ has a structure of a smooth closed manifold of dimension $m+n$, called the \emph{moment-angle manifold} of~$P$.
\end{defi}

Suppose $({\bf{X}},{\bf{A}})=\{(X_i,A_i)\}_{i=1}^{m}$ is a set of topological pairs. A particular case of the following construction, which is the most important one for us here, appeared firstly in the work of Buchstaber and Panov~\cite{bu-pa00-2} and then was studied intensively and generalized in the works of Bahri, Bendersky, Cohen, Gitler~\cite{BBCG}, Grbi\'c and Theriault~\cite{GT}, Iriye and Kishimoto~\cite{IK}, and others.

\begin{defi}
A {\textit{polyhedral product}} is a topological space:
$$
({\bf{X}},{\bf{A}})^K=\bigcup\limits_{I\in K}({\bf{X}},{\bf{A}})^I,
$$
where $({\bf{X}},{\bf{A}})^I=\prod\limits_{i=1}^{m} Y_{i}$ for $Y_{i}=X_{i}$, if $i\in I$, and $Y_{i}=A_{i}$, if $i\notin I$.
Particular cases of a polyhedral product $({\bf{X}},{\bf{A}})^K$ include {\textit{moment-angle-complexes}} $\zk=(\mathbb{D}^2,\mathbb{S}^1)^K$ and {\textit{real moment-angle-complexes}} $\mathcal R_K=(\mathbb{D}^1,\mathbb{S}^0)^K$. 
\end{defi}
A comprehensive survey on homotopy theory of polyhedral products and its relations with the Golod property of face rings can be found in~\cite{GT2}.

Denote by $K_P$ the nerve complex of $P$, i.e., the boundary $\partial P^*$ of the dual simplicial polytope. It can be viewed as a $(n-1)$-dimensional simplicial complex on the set $[m]$, whose simplices are subsets $\{i_1,\ldots,i_k\}$ such that $F_{i_1}\cap\ldots\cap F_{i_k}\ne\varnothing$ in~$P$.
By \cite[Theorem 6.2.4]{TT}, $\zp$ is $\mathbb T^m$-equivariantly homeomorphic to the moment-angle-complex $\mathcal Z_{K_P}$. 

The $\Tor$-groups of $K$ acquire a topological interpretation by means of the following result due to Buchstaber and Panov.

\begin{theo}[{\cite[Theorem 4.5.4]{TT} or \cite[Theorem 4.7]{P}}]\label{zkcoh}
The cohomology algebra of the moment-angle-complex $\mathcal Z_K$ is given
by the isomorphisms
\[
\begin{aligned}
  H^{*,*}(\mathcal Z_K;\ko)&\cong\Tor_{\ko[v_1,\ldots,v_m]}^{*,*}(\ko[K],\ko)\\
  &\cong H\bigl[\Lambda[u_1,\ldots,u_m]\otimes \ko[K],d\bigr]\\
  &\cong \bigoplus\limits_{I\subset [m]}\widetilde{H}^{*}(K_{I};\ko),
\end{aligned}
\]
where bigrading and differential in the cohomology of the differential
bigraded algebra are defined by
\[
  \mathop{\mathrm{bideg}} u_i=(-1,2),\;\mathop{\mathrm{bideg}} v_i=(0,2);\quad
  du_i=v_i,\;dv_i=0.
\]
In the third row, $\widetilde{H}^*(K_{I})$ denotes the reduced simplicial cohomology of the {\textit{induced subcomplex}} $K_{I}$ of $K$ (the restriction of $K$ to $I\subset [m]$). The last isomorphism is the sum of isomorphisms 
$$H^p(\mathcal Z_K)\cong\sum\limits_{I\subset [m]}\widetilde{H}^{p-|I|-1}(K_{I}),$$
and the ring structure is given by the maps
$$
\widetilde{H}^{p-|I|-1}(K_{I})\otimes\widetilde{H}^{q-|J|-1}(K_{J})\to \widetilde{H}^{p+q-|I|-|J|-1}(K_{I\cup J}),\eqno (*)
$$
which are induced by the canonical simplicial maps $K_{I\cup J}\hookrightarrow K_{I}*K_{J}$ (join of simplicial complexes) for $I\cap J=\varnothing$ and zero otherwise.
\end{theo}

Additively the following theorem of Hochster holds.

\begin{theo}[{\cite{Hoch}}]\label{hoch}
For any simplicial complex $K$ on $m$ vertices we have:
$$
\Tor^{-i,2j}_{\ko[v_{1},\ldots,v_{m}]}(\ko[K],\ko)\cong\bigoplus\limits_{J\subset [m],\,|J|=j}\widetilde{H}^{j-i-1}(K_{J};\ko).
$$
\end{theo}
The ranks of the bigraded components of the Tor-algebra 
$$
\beta^{-i,2j}(\ko[K])=\rk_{\ko}\Tor^{-i,2j}_{\ko[v_{1},\ldots,v_{m}]}(\ko[K],\ko)
$$ 
are called the {\textit{bigraded Betti numbers}} of $\ko[K]$ or $K$, when $\ko$ is fixed.
In what follows we need a particular case of the Hochster result for $j=i+1$. One has: 
$$
\beta^{-i,2(i+1)}(P)=\sum\limits_{J\subset [m],|J|=i+1}(\cc(P_J)-1),
$$
where $P_J=\cup_{j\in J}\,F_{j}$ and $\cc(P_J)$ equals the number of connected components of $P_J$.

Due to~\cite[Construction 3.2.8, Theorem 3.2.9]{TT} the Tor-algebra of $K$ acqures a multigrading and the multigraded components can be calculated in terms of induced subcomplexes.
\begin{theo}\label{mgrad}
For any simplicial complex $K$ on $m$ vertices we have:
$$
\Tor^{-i,2J}_{\ko[v_{1},\ldots,v_{m}]}(\ko[K],\ko)\cong\widetilde{H}^{|J|-i-1}(K_{J};\ko),
$$
where $J\subset [m]$ and $\Tor^{-i,2{\bf{a}}}_{\ko[v_{1},\ldots,v_{m}]}(\ko[K],\ko)=0$, if ${\bf{a}}$ is not a $(0,1)$-vector.
\end{theo}

Moreover, if we denote by $R(K)=\Lambda[u_{1},\ldots,u_{m}]\otimes\ko[K]/(v_{i}^{2}=u_{i}v_{i}=0,1\leq i\leq m)$ a graded algebra with the differential $d$ as in Theorem~\ref{zkcoh}, then $R(K)$ also acquires multigrading and the following isomorphism holds:
$$
\Tor^{-i,2{\bf{a}}}_{\ko[v_{1},\ldots,v_{m}]}(\ko[K],\ko)\cong H^{-i,2{\bf{a}}}[R(K),d]
$$
for any simplicial complex $K$.

The author is grateful to Victor Buchstaber and Taras Panov for many helpful discussions and advice. We also thank James Stasheff, Daisuke Kishimoto, Djordje Barali\'c, and the referee of this article for their valuable comments and suggestions on improving the text. 

\section{Nestohedra and graph-associahedra}

We begin with a definition of a family of simple polytopes called nestohedra and state the result of Buchstaber and Volodin on geometric realization of flag nestohedra.

\begin{defi}
Let $[n+1]=\{1,2,\dots,n+1\}$, $n\geq 2$. A {\textit{building set}} on [n+1] is a family of nonempty subsets $B=\{S\subseteq [n+1]\}$, such that: 1) $\{i\}\in B$ for all $1\leq i\leq n+1$, 2) if $S_1\cap S_2\ne\varnothing$, then $S_1\cup S_2\in B$. A building set is called {\textit{connected}} if $[n+1]\in B$.

Then a \textit{nestohedron} is a simple convex $n$-dimensional polytope $P_{B}=\sum\limits_{S\in B}\Delta_{S}$, where in the Minkowski sum one has 
$$
\Delta_{S}=\conv\{{e}_j|\,j\in S\}\subset \mathbb R^{n+1}.
$$ 
Note, that facets of $P_{B}$ are in 1-1 correspondence with proper elements $S$ in $B$ (\cite{FS},\cite[Proposition 1.5.11]{TT}).
\end{defi}

\begin{exam}\label{Bcube,simplex}
If $P$ is a combinatorial $n$-simplex, then the subset of $2^{[n+1]}$ consisting of all the singletons $\{i\},1\leq i\leq n+1$ and the whole set $[n+1]$ gives a connected building set $B$, such that $P=P_{B}$ for any $n\geq 2$. \\
If $P$ is a combinatorial $n$-cube then the following set $B$ consisting of
$$
\{1\},\ldots,\{n+1\},\{1,2\},\{1,2,3\},\ldots,[n+1]
$$
will be a connected building set for $P$ for any $n\geq 2$.
\end{exam}

Any $n$-dimensional nestohedron $P_B$ on a connected building set $B$ can be obtained from an $n$-simplex by a sequence of its face truncations. In order to give the precise statement suppose $B_0\subset B_1$ be building sets on $[n+1]$, and $S\in B_1$. Then define a \emph{decomposition of $S$ into elements
of~$B_0$} as $S=S_1\sqcup\cdots \sqcup S_k$, where $S_j$ are
pairwise nonintersecting elements of $B_0$ and $k$ is minimal
among such disjoint representations of~$S$. One can see easily that this decomposition exists and is unique.

\begin{theo}[{\cite[Lemma 1.5.17, Theorem 1.5.18]{TT}}]\label{Simplextrunc}
Every nestohedron $P_{B}$ corresponding to a connected building
set $B$ can be obtained from a simplex by a sequence of face
truncations.

More precisely, let $B_0\subset B_1$ be connected building sets on $[n+1]$.
Then $P_{B_1}$ is combinatorially equivalent to the polytope
obtained from $P_{B_0}$ by a sequence of truncations at the
faces $G_i=\bigcap_{j=1}^{k_i} F_{S_j^i}$ corresponding to the
decompositions $S^i=S_1^i\sqcup\cdots\sqcup S_{k_i}^i$ of elements
$S^i\in B_1\setminus B_0$, numbered in any order that is
inverse to inclusion (i.e. $S^i\supset S^{i'}\Rightarrow i\le
i'$).
\end{theo}

Buchstaber suggested to call a simple convex $n$-dimensional polytope $P$ a {\textit{2-truncated cube}} if it can be obtained from an $n$-cube by a sequence of cut off some faces of codimension 2 only. It is allowed to cut off any codimension 2 face that we have at the previous step of the sequence of face truncations. 

\begin{exam}
Here is an example of a 3-dimensional 2-truncated cube $\mathcal P$ which we shall use later.
\begin{figure}[h]
\includegraphics[scale=0.5]{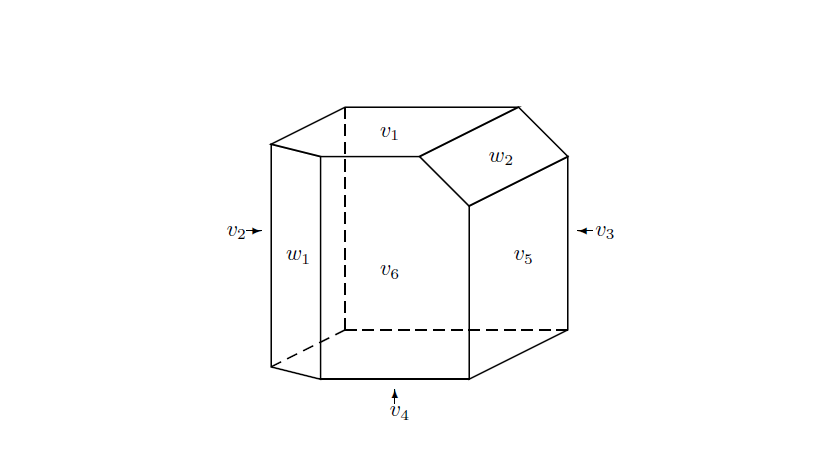}
\caption{A 2-truncated cube $\mathcal P$.}
\label{masseyfig}
\end{figure}

\end{exam}

Then any flag nestohedron can be realized as a 2-truncated cube. The following statement holds.

\begin{theo}[{\cite[Proposition 6.1, Theorem 6.5]{bu-vo11}}]\label{BuVo}
A nestohedron $P_{B}$ is a flag polytope if and only if it is a
2-truncated cube.

More precisely, if $P_{B}$ is a flag polytope, then there exists a
sequence of building sets $B_0 \subset B_1\subset \cdots
\subset B_N=B$, where $P_{B_0}$ is a combinatorial cube,
$B_i=B_{i-1}\cup \{S_i\}$, and $P_{B_i}$ is obtained from
$P_{B_{i-1}}$ by a $2$-truncation at the face $F_{S_{j_1}}\cap
F_{S_{j_2}}\subset P_{B_{i-1}}$ of codimension 2, where $S_i=S_{j_1}\sqcup
S_{j_2}$, and $S_{j_1}, S_{j_2}\in B_{i-1}$.
\end{theo}

The next family of polytopes introduced by M.Carr and S.Devadoss~\cite{CD} are flag nestohedra and, therefore, by Theorem~\ref{BuVo} can be realized as 2-truncated cubes.  

\begin{defi}
\textit{A graphical building set} $B(\Gamma)$ for a (simple) graph $\Gamma$ on the vertex set $[n+1]$ consists of such $S$ that the induced subgraph $\Gamma_{S}$ on the vertex set $S\subset [n+1]$ is a connected graph.\\
Then $P_{\Gamma}=P_{B(\Gamma)}$ is called a \textit{graph-associahedron}.
\end{defi}

\begin{exam}
The following families of graph-associahedra are of particular interest in convex geometry, combinatorics and representation theory. 
\begin{itemize}
\item $\Gamma$ is a complete graph on $[n+1]$.\\
 Then $P_{\Gamma}=Pe^n$ is a \textit{permutohedron}, see Figure~\ref{permfig}.
\begin{figure}[h]
\includegraphics[scale=0.5]{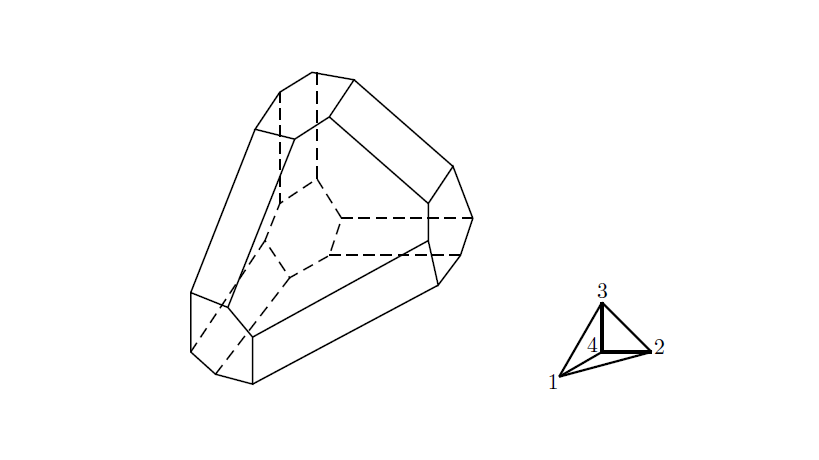}
\caption{3-dimensional permutohedron and the
corresponding graph.}\label{permfig}
\end{figure}

\item $\Gamma$ is a stellar graph on $[n+1]$.\\
 Then $P_{\Gamma}=St^n$ is a \textit{stellahedron}, see Figure~\ref{stelfig}.
\begin{figure}[h]
\includegraphics[scale=0.5]{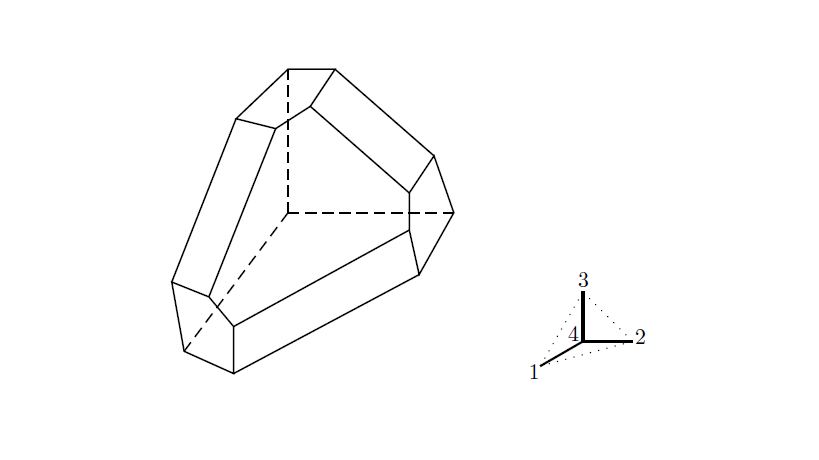}
\caption{3-dimensional stellahedron and the
corresponding graph.}\label{stelfig}
\end{figure}

\item $\Gamma$ is a cycle graph on $[n+1]$.\\
 Then $P_{\Gamma}=Cy^n$ is a \textit{cyclohedron} (or Bott-Taubes polytope~\cite{BT}), see Figure~\ref{cyclfig}.
\begin{figure}[h]
\includegraphics[scale=0.5]{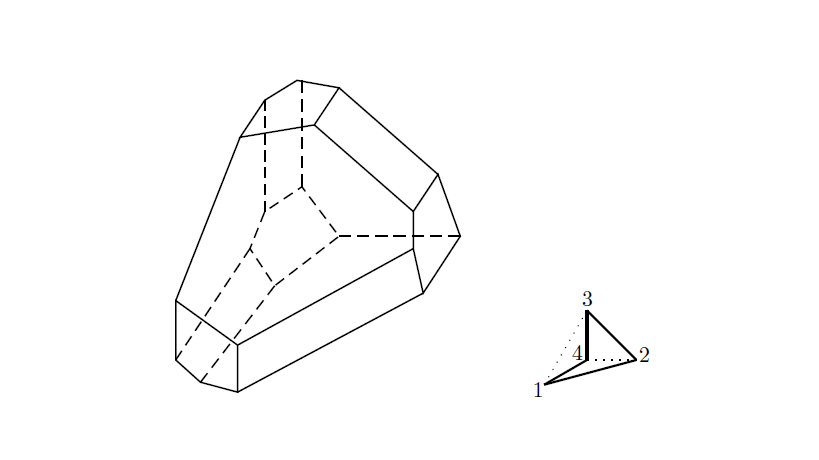}
\caption{3-dimensional cyclohedron and the
corresponding graph.}\label{cyclfig}
\end{figure}

\item $\Gamma$ is a chain graph on $[n+1]$.\\
 Then $P_{\Gamma}=As^n$ is an \textit{associahedron} (or Stasheff polytope~\cite{S}), see Figure~\ref{assfig}.
\begin{figure}[h]
\includegraphics[scale=0.5]{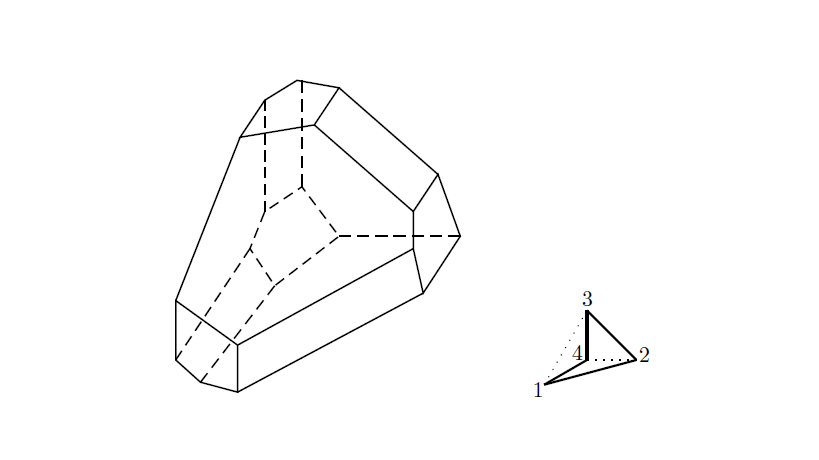}
\caption{3-dimensional associahedron and the
corresponding graph.}\label{assfig}
\end{figure}

\end{itemize}
\end{exam}

In order to determine the nerve complex $K_P$ of a graph-associahedron $P=P_{\Gamma}$ we should describe the combinatorial structure of its face poset.
The following is a reformulation of the general property stated in~\cite[Theorem 1.5.13]{TT}.

\begin{prop}\label{Fposet}
Facets of $P_{\Gamma}$ are in 1-1 correspondence with non maximal connected subgraphs of $\Gamma$.\\ 
Moreover, a set of facets corresponding to such subgraphs $\Gamma_{i_1},\dots, \Gamma_{i_s}$ has a nonempty intersection if and only if:
\begin{itemize}
\item[(1)] For any two subgraphs $\Gamma_{i_k},\Gamma_{i_l}$, either they do not have a common vertex or one is a subgraph of another;

\item[(2)] If any two of the subgraphs 
$\Gamma_{i_{k_1}},\dots,\Gamma_{i_{k_l}},l\geqslant 2$ do not have common vertices, then their union graph is disconnected.
\end{itemize}
\end{prop}

Note, that if $P$ is a permutohedron then its facets $F_{1}$ and $F_{2}$ have a nonempty intersection if and only if the corresponding subgraphs $\Gamma_{1}$ and $\Gamma_{2}$ are subgraphs of one another.

\section{Bigraded Betti numbers of graph-associahedra}

In this section we describe certain bigraded Betti numbers of associahedra $P$ in terms of combinatorics of their graphs $\Gamma$. This approach can be viewed as another argument to prove our previous result~\cite[Theorem 2.9]{Li} and can be used to compute bigraded Betti numbers $\beta^{-i,2(i+1)}(P)$ of all nestohedra. 

\begin{prop}\label{Pe}
The following statements on the additive structure of $H^*(\zp)$ hold.
\begin{itemize}
\item[(a)]
Suppose $P=P_{B_{1}}$ and $Q=P_{B_{2}}$ are $n$-dimensional nestohedra on connected building sets $B_{i}, i=1,2$ and $J\subset B_{1}\subset B_{2}$. Consider the following set: 
$$
\overline{J}=J\sqcup\{S\in{B_{2}\setminus{B_{1}}} |\,\exists S_1\in J, S_{1}\subset S\}.
$$
Then $P^{n}_{J}$ is homeomorphic to $Q^{n}_{\overline{J}}$;
\item[(b)] 
If a graph $\Gamma$ contains a subgraph isomorphic to a complete graph $K_m$ on $m>5$ vertices, or has a vertex of degree more than 5, then there is a 2-torsion element in $H^*(\zp), P=P_{\Gamma}$; moreover, $H^*(\zp)$ may contain any given finite group as a direct summand for some $P=Pe^n$ and $P=St^{n+1}$, if $n\geq 5$. If a graph $\Gamma$ has less than or equal to 5 vertices, then $H^*(\zp), P=P_{\Gamma}$ is torsion free.
\end{itemize}
\end{prop}
\begin{proof}
Let us prove the statement (a). By Theorem~\ref{Simplextrunc} any nestohedron $P_B$ on a connected building set $B\subset 2^{[n+1]}$ can be obtained as a result of a sequence of face truncations starting with a simplex $\Delta^n$. Thus the nerve complex of our nestohedron $K_P=\partial P^*$ can be obtained from a boundary of a simplex as a result of a number of barycentric subdivisions in some of its simplices. Moreover, Theorem~\ref{Simplextrunc} states that the new vertices (barycenters of those simplices) correspond to the decompositions of the elements in $B_{2}\setminus{B_{1}}$ in a disjoint unions of elements of $B_1$. Applying the descriprion of the face poset of $Q$ in~\cite[Theorem 1.5.13]{TT} finishes the proof as any triangulation of a topological space is homeomorphic to the space itself.

Another way to prove this statement is similar to that of the proof in~\cite[Theorem 4.6.4]{Fe}. Indeed, the centers of the geometric realizations of $P$ and $Q$ in $\mathbb{R}^{n+1}$ are Minkowski sums of the centers of their simplices from the definition of a nestohedron. Then we can translate $P$ and $Q$ so that their centers coincide and project the boundary of $P$ onto the boundary of $Q$ outwards from their common center. Obviously, the image of $P_J$ is in $Q^{n}_{\overline{J}}$ and every facet in $Q^{n}_{\overline{J}}$ contains a point in the image of $P_J$. Finally, we make a continuous bijective transformation of the image (on each of the facet in $Q^{n}_{\overline{J}}$) onto the whole $Q^{n}_{\overline{J}}$.

To prove the statement (b), note, that a nerve complex of $Pe^n$ is a barycentric subdivision of a boundary of $\Delta^n$, and the nerve complex of $St^{n+1}$ contains a cone over the nerve complex of $Pe^n$ as an induced subcomplex (see Figure~\ref{stelfig} for the case $n=2$). Therefore, if $K$ has a torsion element in $H^*(K)$, then the same is true for its triangulation, and by Theorem~\ref{zkcoh}, the statement holds with $K$ being a minimal triangulation of $\mathbb{R}P^2$ on 6 vertices, as its barycentric subdivision is an induced subcomplex in a nerve complex of any $Pe^n, n\geq 5$ (see also~\cite{B-M}). The final part follows from the Alexander duality and the Hochster theorem, and holds for any simple polytope of dimension less than 5.   
\end{proof}

In particular, when $B_{2}=2^{[n+1]}$ and $Q$ is a permutohedron, we get the result of Fenn~\cite[Theorem 4.6.4]{Fe}.
In order to describe the bigraded Betti numbers of associahedra combinatorially we introduce the following notion of a special subgraph $\gamma$ in $\Gamma$.

\begin{defi}
Suppose $\Gamma$ is a graph. For any of its connected subgraphs $\gamma$ one can compute the number $i(\gamma)$ of such connected subraphs $\tilde{\gamma}$ in $\Gamma$ that either $\gamma\cap\tilde{\gamma}\neq\varnothing,\gamma,\tilde{\gamma}$ (in this case we say they have a nontrivial intersection) or $\gamma\cap\tilde{\gamma}=\varnothing$, $\gamma\sqcup\tilde{\gamma}$ is a connected subgraph in $\Gamma$. From now on we describe a subgraph in $\Gamma$ as a vertex set meaning that the subgraph consists of its vertices and all edges in $\Gamma$ connecting these vertices (induced subgraph). 
We denote by $i_{max}=i_{max}(\Gamma)$ the maximal value of $i(\gamma)$ over all connected subgraphs $\gamma$ in $\Gamma$. A connected subgraph $\gamma$, on which $i_{max}$ is achieved, will be called a {\textit{special subgraph}}.
\end{defi}

\begin{exam}
On the Figure~\ref{assfig} we have 3 special subgraphs: $\{1,2\},\{1,4\}$ and $\{2,3\}$. The number $i_{max}$ is equal to 4 and is achieved on $\gamma=\{1,2\}$ the graphs $\tilde{\gamma}$ are: $\{3\},\{4\},\{1,4\},\{2,3\}$ (the latter two intersect $\gamma$ nontrivially).
\end{exam}

The following statement for the bigraded Betti numbers of the type $\beta^{-i,2(i+1)}(P)$ for associahedra $P$ holds.

\begin{theo}\label{bigraded}
Let $P=P_{\Gamma}$ be an associahedron of dimension $n\geq 3$. Then for $i>i_{max}(\Gamma)$ one has:
$$
\beta^{-i,2(i+1)}(P)=0.
$$
Denote the number of special subgraphs in $\Gamma$ by $s$. Let $\omega=-i_{max}, 2(i_{max}+1)$. Then
$$
\beta^{\omega}(P)=s.
$$
\end{theo}
\begin{proof}
By Theorem~\ref{hoch} and Proposition~\ref{Fposet} it is sufficient to prove the following:
\begin{itemize}
\item[(a)] We have: $\cc(P_J)\leq 2$ if $|J|>i_{max}$. In the latter case, if $P_{J}=P_{J_{1}}\sqcup P_{J_{2}}$ with $|J_{1,2}|\geq 2$, then there exists another $J'\subset B(\Gamma)$ such that $P_{J'}=P_{J'_{1}}\sqcup P_{J'_{2}}$ with $|J'_{1}|=1$ and $|J'|>|J|$.
 
\item[(b)] Suppose $\cc(P_J)=2, P_{J}=P_{J_{1}}\sqcup P_{J_{2}}, |J|>i_{max}$. Then either $|J_{1}|=1$ or $|J_{2}|=1$. Moreover, if $|J|=i_{max}+1, |J_{1}|=1$ then $J_{1}$ consists of a special subgraph of $\Gamma$ and $J_{2}$ consists of all the $i_{max}$ connected subgraphs in $\Gamma$ determined in the definition of a special graph above.
 
\item[(c)] Suppose $|J|>i_{max}+1$. Then $\cc(P_J)=1$.
\end{itemize}

For an associahedron $As^{n}$ the statement (a) follows from~\cite[Lemmata 2.13, 2.14]{Li}, the statement (b) follows from~\cite[Lemmata 2.15, 2.16]{Li} and the statement (c) follows from~\cite[Lemma 2.17]{Li}.  
\end{proof}

\begin{rema}
Using Propostion~\ref{Fposet} one can see easily that Theorem~\ref{bigraded} states the last nonzero bigraded Betti number $\beta^{\omega}(P)$ in the sequence of $\beta^{-i,2(i+1)}(P),1\leq i\leq m-n$ to be achieved precisely on $P_J$ which is a union of a facet of $P$ corresponding to a special subgraph in $\Gamma$ and all the facets of $P$ that do not intersect this facet. All the $P_J$ with a greater cardinality $|J|$ of $J$ are connected spaces in $\mathbb{R}^n$. An argument similar to that in the proof of Theorem~\ref{bigraded} shows the same holds for a permutohedron $Pe^n, n\geq 3$ and applying Proposition~\ref{Pe} one can get the same result for any graph-associahedron on a connected graph $\Gamma$.
\end{rema}

As an application of Theorem~\ref{bigraded} the values of $i_{max}(\Gamma)$ and $s$ can be computed explicitly in terms of the combinatorics of the graph $\Gamma$. Using induction on the polytope dimension $n$ for combinatorial enumerations in $\Gamma$ it can be seen that a special graph $\gamma$ is a path graph in $\Gamma$ on either $[\frac{n+1}{2}]$ or $[\frac{n}{2}]+1$ vertices. This follows also from the proof of~\cite[Theorem 2.9]{Li} where the special graphs correspond to the longest diagonals in a regular $(n+3)$-gon $G$ and the numbers of the vertices in such a graph are the numbers of vertices of $G$ lying in one of the open halves of $G$ divided by the diagonal. Thus, we get the following result, cf.~\cite[Theorem 2.9]{Li}

\begin{coro}
For an associahedron $P_{\Gamma}$ of dimension $n\geq 3$ one has the following values of $i_{max}=q(n)$ and $s$.  

\begin{align*}
  &\beta^{-q,2(q+1)}(As^{n})=\begin{cases}
  n+3,&\text{if $n$ is even;}\\
  \frac{n+3}{2},&\text{if $n$ is odd;}
  \end{cases}\\
  &\beta^{-i,2(i+1)}(As^{n})=0\quad\text{for }i\ge{q+1},
\end{align*}
where $q=q(n)$ is: 
\begin{align*}
&q=q(n)=\begin{cases}
  \frac{n(n+2)}{4},&\text{if $n$ is even;}\\
  \frac{(n+1)^2}{4},&\text{if $n$ is odd.}
  \end{cases}
\end{align*}
\end{coro}

As graph-associahedra are flag polytopes, we can apply the previous result to studying the loop homology algebra $H_{*}(\Omega\zp)$ for associahedra $P$. Namely, due to~\cite[Theorem 4.3]{G-P-T-W} the minimal number of multiplicative generators of $H_{*}(\Omega\zp)$ is equal to $\sum\limits_{i=1}^{m-n}\beta^{-i,2(i+1)}(P)$. Then Theorem~\ref{bigraded}
gives us lower bounds for the number of multiplicative generators in the Pontryagin algebra of $\zp$.

\section{Massey products}

In this section we prove the main result of this article concerning Massey higher products in $H^{*}(\mathcal Z_P)$ (Theorem~\ref{mainMassey}), first stated in~\cite{Li2}, and a criterion for a nontrivial triple Massey product of 3-dimensional classes to exist in $H^{*}(\mathcal Z_{P_{\Gamma}})$ (Proposition~\ref{assocMassey}). 
We first prove the statement on triple Massey products in the graph-associahedron $P_{\Gamma}$ case, where $\Gamma$ is an arbitrary (possibly disconnected) graph.

Let us state the following theorem due to Denham and Suciu which gives a combinatorial criterion for a simplicial complex $K$ to produce a nontrivial triple Massey product of 3-dimensional classes in $H^{*}(\mathcal Z_K)$.

\begin{theo}[{\cite[Theorem 6.1.1]{DS}}]\label{DenSuc}
The following are equivalent:
\begin{itemize}
\item[(1)] There exist cohomology classes $\alpha_{i}\in H^{3}(\mathcal Z_K), i=1,2,3$ for which $\langle\alpha_{1},\alpha_{2},\alpha_{3}\rangle$ is defined and nontrivial.
\item[(2)] The underlying graph (1-skeleton) of $K$ contains an induced subgraph isomorphic to one of the five graphs in Figure~\ref{DenhamSuciu}.
\end{itemize}
Moreover, all Massey products arising in this fashion are decomposable.
\end{theo} 

\begin{figure}[h]
\includegraphics[scale=0.5]{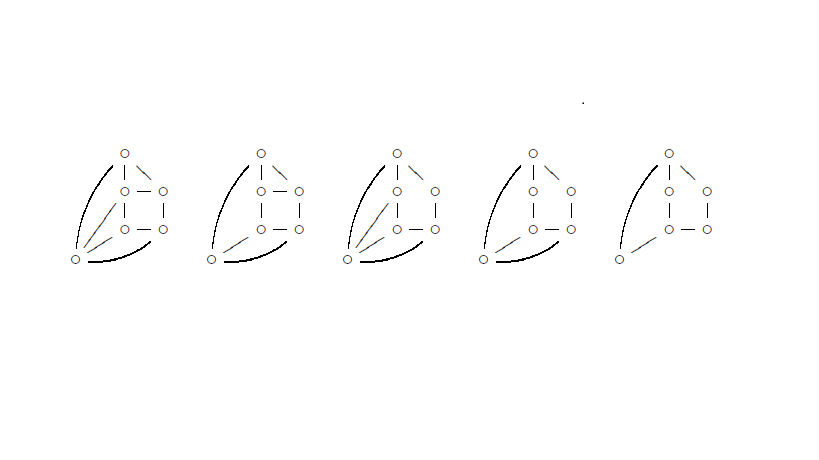}
\caption{The five obstruction graphs.}
\label{DenhamSuciu}
\end{figure}

Applying Theorem~\ref{DenSuc} to the graph-associahedra case, we get the following result.

\begin{prop}\label{assocMassey}
There is a nontrivial triple Massey product $\langle\alpha_{1},\alpha_{2},\alpha_{3}\rangle$ of 3-dimensional cohomology classes $\alpha_{i}\in H^{3}(\mathcal Z_{P_{\Gamma}})$ for $i=1,2,3$ if and only if there is a connected component of $\Gamma$ on $m\geq 4$ vertices which is different from a complete graph $K_{4}$. 
\end{prop}
\begin{proof}
We start with a connected graph $\Gamma$ case. Suppose the number of vertices in $\Gamma$ is less than 4. Then $P_{\Gamma}$ is either a point, a segment, a pentagon or a hexagon. The corresponding moment-angle manifold $\zp$ is either a disk $D^{2}$, a sphere $S^3$, or a connected sum of products of spheres, respectively (\cite{McG},\cite{B-M}). These manifolds are formal spaces, therefore, there are no nontrivial higher Massey products in $H^{*}(\zp)$. 

Suppose there are 4 vertices in $\Gamma$. There are 6 combinatorially different connected graphs $\Gamma$ on 4 vertices, thus, giving 6 combinatorially different 3-dimensional graph-associahedra $P_{\Gamma}$. If $\Gamma$ is a complete graph $K_4$ then $P=P_{\Gamma}$ is a permutohedron and the boundary of its dual simplicial polytope $K=K_P$ is combinatorially equivalent to a barycentric subdivision of a boundary of a 3-simplex. As there are no induced 5-cycles in $K$ and the first two graphs in Figure~\ref{DenhamSuciu} can not also be induced graphs in $K$, by Theorem~\ref{DenSuc} there are no nontrivial triple Massey products in $H^{*}(\zp)$. On the other hand, using Fugures~\ref{permfig}, \ref{stelfig}, \ref{cyclfig}, \ref{assfig} and Theorem~\ref{Simplextrunc} one can check easily that the third (middle) of the five graphs in the Figure~\ref{DenhamSuciu} is an induced subgraph in the underlying graph (1-skeleton) of $K_P$ for $P$ being any of the other five 3-dimensional graph-associahedra on a connected graph with 4 vertices. The case of a connected graph on 4 vertices now holds from Theorem~\ref{DenSuc}.    

Suppose now, that $\Gamma$ is a connected graph on more than 4 vertices. Using induction on the number of edges in $\Gamma$ we get an induced subgraph $\gamma$ in $\Gamma$ on 4 vertices. Using Proposition~\ref{Fposet} the induced subcomplex in $K_P, P=P_{\Gamma}$ on the vertex set corresponding to all connected subgraphs in $\gamma\neq K_{4}$ will give us a nontrivial triple Massey product in $H^{*}(\zp)$ by the argument above. On the other hand, if any connected subgraph on 4 vertices in $\Gamma$ is a complete graph $K_4$ then $\Gamma$ is a complete graph $K_{n+1}$. Indeed, consider two different vertices $\alpha$ and $\beta$ in $\Gamma$. Then there is a connected subgraph containing them in $\Gamma$. Such a graph $\gamma$ with a minimal number of edges will obviously be a path between $\alpha$ and $\beta$. If it has more than 2 edges then it has more than 3 vertices and thus contains $K_4$ as an induced graph on some 4 of its vertices, thus $\gamma$ being not minimal (any pair of vertices in $K_4$ is connected by one edge). Similarly, if $\gamma$ has 2 edges then one of its 3 vertices is conected to another vertex of $\Gamma$ (as $\Gamma$ has more than 4 vertices and is connected) and we get $K_4$ as an induced subgraph. So, $\gamma$ is not minimal again. Thus, $\gamma$ has one edge, i.e., $\alpha$ and $\beta$ are connected by an edge in $\Gamma$ and $\Gamma$ is a complete graph.   

It remains to consider the case when $\Gamma$ is a complete graph $K_{n+1}, n\geq 4$ and $P=P_{\Gamma}$ is a permutohedron. Note, that $K_{Q}$ is an induced subcomplex in $K_{P}$ for any such $P$ when $Q$ is a 4-dimensional permutohedron. Consider the graph $\Gamma=K_{5}$ for $Q$ and an induced subcomplex of $K_Q$ on the following vertices: 
$$
\{1\},\{2\},\{1,3\},\{1,2,4\},\{1,2,3,4\},\{1,2,3,5\}. 
$$ 
One can see easily that its graph is the first (left) graph in Figure~\ref{DenhamSuciu}. By Theorem~\ref{DenSuc} and Theorem~\ref{zkcoh} (formula (*)), any permutohedron $P$ of dimension 4 and greater gives us a nontrivial triple Massey product in $H^{*}(\zp)$.

Finally, the case of a disconnected graph $\Gamma$ follows from Proposition~\ref{Fposet} and Theorem~\ref{zkcoh} and the connected graph case as if two graphs $\Gamma_1$ and $\Gamma_2$ are disjoint then for their union graph $\Gamma$ one gets: $P_{\Gamma}=P_{\Gamma_{1}}\times P_{\Gamma_{2}}$ and the moment-angle map $\mathcal Z$ preserves products~\cite[Ch. 4]{TT}. This finishes the proof. 
\end{proof}

\begin{rema}
Note, that each of the six 3-dimensional graph-associahedra $P=P_{\Gamma}$ mentioned above is a 2-truncated cube and, moreover, $Pe^3$ can be obtained from $Cy^3$ by cut off its 2 non-adjacent edges, if realized as a simplex truncation (see Theorem~\ref{Simplextrunc} and Figures~\ref{cyclfig}, \ref{permfig}). As $\zp$ for $P=I^n$ is a product of spheres and, therefore, is a formal manifold, it follows that a nontrivial higher Massey product in $H^{*}(\zp)$ can either appear or vanish after a (codimension 2) face truncation (or, equivalently, after a stellar subdivision in the dual simplicial sphere $K_P$).  
\end{rema}

\begin{exam}
Consider $P=Pe^3$, see Figure~\ref{permfig}. It has $n=3$ and $m=14$, let us label its facets by the numbers $1,\ldots,14$ such that the bottom and upper 6-gon facets are 1 and 14 resp., the bottom facets are labeled by $2,\ldots,7$ and the upper facets are labeled by $8,\ldots,13$, both clockwisely.

Consider the following 3-cocycles: 
$$
a_{1}=v_{1}u_{14}, a_{2}=v_{6}u_{10}, a_{3}=v_{8}u_{4}, a_{4}=v_{2}u_{12}.
$$    
They correspond to 4 pairs of parallel facets of $P$ if realized as a result of face truncations from $\Delta^3$. Suppose they are representatives of the cohomology classes $\alpha_{i}\in H^{3}(\zp)$, that is, $\alpha_{i}=[a_i]$ for $i=1,\ldots,4$. 

Then we get the following defining system $A$ (see~\cite{Kr}) for the Massey 4-product $\langle\alpha_{1},\alpha_{2},\alpha_{3},\alpha_{4}\rangle$ (up to signs):
$$
a_{13}=v_{6}u_{1}u_{14}u_{10}, a_{24}=v_{6}u_{10}u_{8}u_{4}, a_{35}=v_{2}u_{8}u_{4}u_{12},
$$
$$
a_{14}=v_{6}u_{1}u_{8}u_{4}u_{10}u_{14}, a_{25}=0,
$$
so $0\in\langle\alpha_{1},\alpha_{2},\alpha_{3},\alpha_{4}\rangle$. Thus the two 3-products $\langle\alpha_{1},\alpha_{2},\alpha_{3}\rangle$ and $\langle\alpha_{2},\alpha_{3},\alpha_{4}\rangle$ are defined and vanish simultaneously and the 4-product $\langle\alpha_{1},\alpha_{2},\alpha_{3},\alpha_{4}\rangle$ is defined and trivial. 
\end{exam}

\begin{rema}
Note, that the same calculation works in full generality, namely: if $P=Pe^n, n\geq 2$ and the classes $\alpha_{i}\in H^{3}(\zp),1\leq i\leq n+1$ are represented by $(n+1)$ pairs of the parallel permutohedra facets (cf. Figure~\ref{permfig}), then $\langle\alpha_{1},\ldots,\alpha_{n+1}\rangle$ is defined and trivial. Similarly, if $P=St^n, n\geq 2$ and the classes $\alpha_{i}\in H^{3}(\zp),1\leq i\leq n$ are represented by $n$ pairs of the parallel stellahedra facets (cf. Figure~\ref{stelfig}), then $\langle\alpha_{1},\ldots,\alpha_{n}\rangle$ is defined and trivial.
\end{rema}

We next consider a particular family of 2-truncated $n$-cubes $\mathcal P$, one for each dimension $n$, for which $\mathcal Z_{\mathcal P}$ has a nontrivial Massey product of order $n$. 

\begin{defi}
Suppose $I^{n}=[0,1]^n, n\geq 2$ is an $n$-dimensional cube with facets $F_{1},\ldots,F_{2n}$, such that $F_{i},1\leq i\leq n$ contain the origin 0, a unit inner normal vector to $F_{i},1\leq i\leq n$ is $(0,\ldots,1,\ldots,0)$ with 1 in the $i$th position, $F_{i}$ and $F_{n+i}$, $1\leq i\leq n$ being parallel. Then we define $\mathcal P$ as a result of a consecutive cut of faces of codimension 2 from $I^{n}$, having the following Stanley-Reisner ideal:
$$
I=(v_{k}v_{n+k+i},0\leq i\leq n-2,1\leq k\leq n-i,\ldots),
$$
where $v_{i}$ correspond to $F_{i}$, $1\leq i\leq 2n$ and in the dots are the monomials corresponding to the new facets (i.e., facets obtained after performing truncations). This determines uniquely the combinatorial type of $\mathcal P$. 
\end{defi}

\begin{exam}
For $n=2$ we get a 2-dimensional cube (the square) $\mathcal P$ and its Stanley-Reisner ideal is the following one:
$$
I=(v_{1}v_{3},v_{2}v_{4}).
$$
For $n=3$ we get a simple polytope $\mathcal P$ from Figure~\ref{masseyfig}, for which $K=K_{\mathcal P}$ is a simplicial complex with a nontrivial triple Massey product in $H^{*}(\zk)$ due to the result of Baskakov, see~\cite{BaskM}. Moreover, using the computer software {\textit{Plantri}} it can be seen that $K$ is the only one of the 14 combinatorially different 2-spheres with 8 vertices giving nontrivial higher Massey products in $\mathcal Z_K$, see~\cite{DS}.\\
The Stanley-Reisner ideal of $\mathcal P$ can be written as follows, see Figure~\ref{masseyfig}:
$$
I=(v_{1}v_{4},v_{2}v_{5},v_{3}v_{6},v_{1}v_{5},v_{2}v_{6},w_{1}v_{3},w_{1}v_{5},w_{2}v_{2},w_{2}v_{4},w_{1}w_{2}).
$$  
\end{exam}

\begin{rema}
The 2-truncated cube $\mathcal P$ is not a graph-associahedron as its number of facets $f_{0}(\mathcal P)=\frac{n(n+3)}{2}-1<f_{0}(As^{n})=\frac{n(n+3)}{2}$, see~\cite[Theorem 9.2]{bu-vo11}.
However, we can easily construct the building set $B$ for $\mathcal P$ on the vertex set $[n+1]$ by identifying $F_{i}$ with $\{1,\ldots,i\}$ for $1\leq i\leq n$ and identifying $F_{i}$ with $\{i-n+1\}$ for $n+1\leq i\leq 2n$. Then, by Theorem~\ref{BuVo}, we consecutively cut the following faces:
$$
\{1\}\sqcup\{3\}, \{1,2\}\sqcup\{4\},\ldots,\{1,\ldots,n-1\}\sqcup\{n+1\}\\
$$
$$
\cdots\\
$$
$$
\{1\}\sqcup\{n\}, \{1,2\}\sqcup\{n+1\}. 
$$
Thus, $\mathcal{P}=P_{B}$ for the the building set $B$ consisting of the building set $B_0$ of an $n$-cube from Example~\ref{Bcube,simplex}, the above subsets of $[n+1]$ and all the subsets of $[n+1]$ which are the unions of nontrivially intersecting elements in $B$. 
\end{rema}

\begin{theo}\label{mainMassey}
Let $\alpha_i\in H^{3}(\mathcal Z_{\mathcal P})$ be represented by a 3-cocycle $v_{i}u_{n+i}$ for $1\leq i\leq n$ and $n\geq 2$. Then all Massey products of consecutive elements from $\alpha_{1},\ldots,\alpha_{n}$ are defined and the whole $n$-product $\langle\alpha_{1},\ldots,\alpha_{n}\rangle$ is nontrivial.
\end{theo}
\begin{proof}
Let us prove the theorem by induction on $n$. The base case $n=2$ is trivial: $\alpha_1$ and $\alpha_2$ are the classes of 3-dimensional spheres in $\mathcal Z_{\mathcal P}\cong S^{3}\times S^{3}$ and their cup-product (i.e., Massey 2-product) is the dual to the fundamental class of $\mathcal Z_{\mathcal P}$. 

We first note that all Massey products of orders less than $n$ vanish simultaneously in $H^{*,*}(\mathcal Z_{\mathcal P})\cong H\bigl[\Lambda[u_1,\ldots,u_m]\otimes \ko[\mathcal P],d\bigr]$, i.e., contain coboundaries. Starting with the representing cocycles $v_{i}u_{i+n}$ of $\alpha_i$ it can be seen by induction on the dimension $n$ of $\mathcal P$ that if a defining system $C$ for the $n$-product $\langle\alpha_{1},\ldots,\alpha_{n}\rangle$ can be extended from $i$th diagonal of the matrix $C$ to its $(i+1)$th diagonal for all $2\leq i\leq n$ then $c_{lm}, m-l=i\geq 2$ have either a form $v_{k}u_{j_1}\ldots u_{j_{2i-1}}$ or a form $v_{k}u_{j_1}\ldots u_{j_{2i-1}}+d(u_{k}u_{j_1}\ldots u_{j_{2i-1}})$ (up to the signs). The latter can be checked as the differential in the cohomology algebra preserves multigrading and by using the codimension 2 face cuts from the definition of $\mathcal P$ (see also the Example below).

Then the Massey $n$-product $\langle\alpha_{1},\ldots,\alpha_{n}\rangle$ is defined and any cohomology class belonging to it lies in the multigraded component  $H^{-(2n-2),(2,\ldots,2,0,\ldots,0)}(\mathcal Z_{\mathcal P})$ of the moment-angle manifold $\mathcal Z_{\mathcal P}$ with one of its representatives being the class of the cocycle $v_{1}v_{2n}u_{2}\ldots u_{2n-1}$. Up to sign we have the following equality for any representative $c$ of an element in $\langle\alpha_{1},\ldots,\alpha_{n}\rangle$ for any defining system $C$, see~\cite{Kr}:
$$
c=d(c_{1,n+1})-(-1)^{3}v_{1}u_{n+1}c_{2,n+1}-\overline{c}_{1,3}c_{3,n+1}-\ldots-\overline{c}_{1,n}v_{n}u_{2n},
$$
where $(n+1)\times (n+1)$-matrix $C$ is upper triangular with zeros on the diagonal and $c_{i,i+1}=-v_{i}u_{n+i}$ for $1\leq i\leq n$, such that the following condition holds:
$$
cE_{1,n+1}=d(C)-\overline{C}\cdot{C}
$$ 
and $\overline{c}_{ij}=(-1)^{|c_{ij}|}c_{ij}$ depends on the degree $|c_{ij}|$ of a matrix element $c_{ij}$. 

By definition of higher Massey operations (see~\cite{Kr}), one has: $d(c_{2,n+1})$ is a representative in $\langle\alpha_{2},\ldots,\alpha_{n}\rangle$ and $d(c_{1,n})$ is a representative in $\langle\alpha_{1},\ldots,\alpha_{n-1}\rangle$. To prove that $v_{1}v_{2n}u_{2}\ldots u_{2n-1}$ is the only representing cocycle for the $n$-product we use induction on $n$, the representing monomials for the indeterminacies and the multigrading in $H^{*}(\mathcal Z_{\mathcal P})$, see Theorem~\ref{mgrad}. For instance, the indeterminacy for the first of the $(n-1)$-products above lies in the multigraded component of $v_{2}u_{3}\ldots u_{n}u_{n+2}\ldots u_{2n}$ and the only cocycle there is the coboundary $d(u_{2}\ldots u_{n}u_{n+2}\ldots u_{2n})$. As for the second $(n-1)$-product above, the indeterminacy lies in the multigraded component of $v_{1}u_{2}\ldots u_{n-1}u_{n+1}\ldots u_{2n-1}$ and the only cocycle there is the coboundary $d(u_{1}\ldots u_{n-1}u_{n+1}\ldots u_{2n-1})$.

Thus, the Massey $n$-product $\langle\alpha_{1},\ldots,\alpha_{n}\rangle$ is defined and nontrivial, consisting only of the cohomology class of $v_{1}v_{2n}u_{2}\ldots u_{2n-1}$. 
\end{proof}

\begin{rema}
Note, that the nontrivial $n$-product constructed above is decomposable. 
Namely, one has: $[v_{1}v_{2n}u_{2}\ldots u_{2n-1}]=\pm[v_{1}u_{n+1}\ldots u_{2n-1}]\cdot [v_{2n}u_{2}\ldots u_{n}]$.
\end{rema}

\begin{exam}
Consider the case $n=4$. Then the Stanley--Reisner ideal of $\mathcal P$ is 
$$
I=(v_{1}v_{5},v_{2}v_{6},v_{3}v_{7},v_{4}v_{8},v_{1}v_{6},v_{2}v_{7},v_{3}v_{8},v_{1}v_{7},v_{2}v_{8},\ldots)
$$
and the cohomology classes $\alpha_{i},1\leq i\leq 4$ are represented by the cocycles $a_{i}=v_{i}u_{4+i},1\leq i\leq 4$. One has (up to sign):
$$
a_{1}a_{2}=d(v_{1}u_{2}u_{5}u_{6})=d(c_{1,3}),
$$
$$
a_{2}a_{3}=d(v_{2}u_{3}u_{6}u_{7})=d(c_{2,4}),
$$
$$
a_{3}a_{4}=d(v_{2}u_{4}u_{7}u_{8})=d(c_{3,5}).
$$
Then one has the following cocycle representing a class in $\langle\alpha_{1},\alpha_{2},\alpha_{3}\rangle$ (here the Massey 2-product of $a$ and $b$ is equal to $\overline{a}\cdot{b}$, $\overline{a}=(-1)^{|a|}a$): 
$$
v_{1}u_{5}\cdot (-v_{2}u_{3}u_{6}u_{7})-v_{1}u_{2}u_{5}u_{6}\cdot v_{3}u_{7}=d(v_{1}u_{2}u_{3}u_{5}u_{6}u_{7})=d(c_{1,4})
$$ 
and the following cocycle representing a class in $\langle\alpha_{2},\alpha_{3},\alpha_{4}\rangle$:
$$
v_{2}u_{6}\cdot (-v_{3}u_{4}u_{7}u_{8})-v_{2}u_{3}u_{6}u_{7}\cdot v_{4}u_{8}=
d(v_{2}u_{3}u_{4}u_{6}u_{7}u_{8})=d(c_{2,5}).
$$
Alternatively, one has (up to sign):
$$
a_{1}a_{2}=d(v_{2}u_{1}u_{5}u_{6}-v_{5}u_{2}u_{1}u_{6}+v_{6}u_{2}u_{1}u_{5})=d(c_{1,3}),\\
$$
$$
a_{2}a_{3}=d(v_{3}u_{2}u_{6}u_{7}-v_{6}u_{3}u_{2}u_{7}+v_{7}u_{3}u_{2}u_{6})=d(c_{2,4}),\\
$$
$$
a_{3}a_{4}=d(v_{4}u_{2}u_{7}u_{8}-v_{7}u_{4}u_{2}u_{8}+v_{8}u_{4}u_{2}u_{7})=d(c_{3,5}).
$$
The representing cocycle for $\langle\alpha_{1},\alpha_{2},\alpha_{3}\rangle$ will be $d(v_{3}u_{1}u_{2}u_{5}u_{6}u_{7})=d(c_{1,4})$ and for $\langle\alpha_{2},\alpha_{3},\alpha_{4}\rangle$ one gets: $d(v_{4}u_{2}u_{3}u_{6}u_{7}u_{8})=d(c_{2,5})$. 

Thus, the Massey products $\langle\alpha_{1},\alpha_{2},\alpha_{3}\rangle$ and $\langle\alpha_{2},\alpha_{3},\alpha_{4}\rangle$ vanish simultaneously and the 4-product $\langle\alpha_{1},\alpha_{2},\alpha_{3},\alpha_{4}\rangle$ is defined.
More precisely, the representing cocycle $c$ for $\langle\alpha_{1},\alpha_{2},\alpha_{3},\alpha_{4}\rangle$ is equal to:
$$
d(c_{1,5})-\overline{a}_{1}c_{2,5}-\overline{c}_{1,3}c_{3,5}-\overline{c}_{1,4}a_{4}.
$$ 
Considering the multigrading in $H^{*}(\zp)$ it is easy to see that the latter 4-fold product consists of the only class with a representative (up to sign) $v_{1}v_{8}u_{2}\ldots u_{7}$ in $H^{-6,(2,\ldots,2,0,\ldots,0)}(\zp)\subset H^{-6,16}(\zp)\subset H^{10}(\zp)$, where $\zp$ is a closed smooth 17-dimensional manifold. 

Finally, one has: $[v_{1}v_{8}u_{2}\ldots u_{7}]=-[v_{1}u_{5}u_{6}u_{7}]\cdot [v_{8}u_{2}u_{3}u_{4}]$.
\end{exam}

Using Theorem~\ref{mainMassey} we can construct a smooth closed 2-connected manifold $M$ with a compact torus action, such that there are nontrivial higher Massey products of any prescribed orders $n_{1},\ldots,n_{r}, r\geq 2$ in $H^{*}(M)$. Namely, consider the building sets $B_{i},1\leq i\leq r$ for $\mathcal P^{n_i},1\leq i\leq r$. Let $M=\zp$, where $P=P_{B'}, B'=B(B_{1},\ldots,B_{r})$ (\cite[Construction 1.5.19]{TT}) and $B$ be a connected building set of a $(r-1)$-dimensional cube. Then $P$ is a flag polytope combinatorially equivalent to $I^{r-1}\times\mathcal P^{n_1}\times\ldots\times\mathcal P^{n_r}$ (\cite[Lemma 1.5.20]{TT}) and $H^{*}(\zp)$ contains nontrivial Massey products of orders $n_{i},1\leq i\leq r$ as the functor $\mathcal Z$ preserves products for simple polytopes. Note, that $P=P_{B'}$ is still a flag nestohedron and, therefore, can be realized as a 2-truncated cube.


\begin{thebibliography}{99}

\bibitem{BBCG} Anthony Bahri, Martin Bendersky, Frederick R. Cohen, and Samuel Gitler. \emph{The polyhedral product functor: a method of computation for moment-angle complexes, arrangements and related spaces}, Adv. Math. {\bf{225}} (2010), no. 3, 1634--1668.

\bibitem{BaskM} I.\,V.\,Baskakov. \textit{Massey triple products in the cohomology of moment-angle complexes}, Russian Math. Surveys, 58:5 (2003), 1039--1041.

\bibitem{B-M}
Fr\'ed\'eric Bosio and Laurent Meersseman. \emph{Real quadrics in
$\mathbb C^n$, complex manifolds and convex polytopes.} Acta
Math.~\textbf{197} (2006), no.~1, 53--127.

\bibitem{BT}
Raul Bott and Clifford Taubes. \textit{On the self-linking of knots. Topology and physics}. J.Math.Phys. {\bf{35}} (1994), no. 10, 5247--5287.

\bibitem{TT} Victor M. Buchstaber and Taras E. Panov. \emph{Toric Topology}, Mathematical Surveys and Monographs, 204, American Mathematical Society, Providence, RI, 2015.

\bibitem{bu-pa00-2}
Victor M.~Buchstaber and Taras E.~Panov. \emph{Torus actions,
combinatorial topology and homological algebra}. Uspekhi Mat.
Nauk~{\bf 55} (2000), no.~5, 3--106 (Russian). Russian Math.
Surveys~{\bf 55} (2000), no.~5, 825--921 (English translation).

\bibitem{bu-vo11} Victor M. Buchstaber, Vadim D. Volodin. \emph{Precise upper and lower bounds for nestohedra}. Izv. Ross. Akad. Nauk, Ser. Mat.~{\bf 75} (2011), no.~6, 17--46 (Russian); Izv. Math.~{\bf 75} (2011), no.~6
(English translation).

\bibitem{CD} Michael P. Carr and Satyan L. Devadoss. \textit{Coxeter complexes and graph-associahedra}, Topology Appl. {\bf{153}} (2006), no. 12, 2155--2168.

\bibitem{DS} Graham Denham and Alexander I. Suciu. \textit{Moment-angle complexes, monomial ideals, and Massey products}, Pure and Applied Mathematics Quarterly, 3(1) (2007), (Robert MacPherson special issue, part 3), 25--60.

\bibitem{FS} Eva Maria Feichtner and Bernd Sturmfels. \emph{Matroid polytopes, nested sets and Bergman fans}, Port. Math. (N.S.) {\bf 62} (2005), no. 4, 437--468.

\bibitem{Fe} A.~Fenn.\emph{On families of nestohedra}, PhD thesis, Manchester University, 2010. 

\bibitem{GT} Jelena Grbi\'c and Stephen Theriault. \emph{The homotopy type of the complement of a coordinate subspace arrangement}, Topology {\bf{46}} (2007), no. 4, 357--396.

\bibitem{GT2} Jelena Grbi\'c and Stephen Theriault. \emph{Homotopy theory in toric topology}, Russian Mathematical Surveys, 71:2 (2016), 185--251.

\bibitem{G-P-T-W} Jelena Grbic, Taras Panov, Stephen Theriault and Jie Wu. \emph{Homotopy types of moment-angle complexes for flag complexes}, Trans. Amer. Math. Soc., {\bf{368}} (2016), no.9, 6663-6682; arXiv:1211.0873.

\bibitem{Hoch} M.\,Hochster, \textit{Cohen-Macaulay rings, combinatorics, and simplicial complexes}, in Ring theory, II (Proc. Second
Conf.,Univ. Oklahoma, Norman, Okla., 1975),  Lecture Notes in Pure
and Appl. Math., V. 26, 171--223, Dekker, New York, 1977.

\bibitem{IK} Kouyemon Iriye and Daisuke Kishimoto. \emph{Decompositions of polyhedral products for shifted complexes}, Adv. Math. {\bf{245}} (2013), 716--736.

\bibitem{Kr} David Kraines. \emph{Massey higher products}, Trans. Amer. Math. Soc., {\bf{124}} (1966), 431--449. 

\bibitem{Li} Ivan Yu. Limonchenko. \emph{Bigraded Betti numbers of certain simple polytopes}, Mathematical Notes, 94:3 (2013), 351-363.

\bibitem{Li2} Ivan Yu. Limonchenko. \emph{Massey products in cohomology of moment-angle manifolds for 2-truncated cubes}, Russian Mathematical Surveys, 71:2 (2016), 376--378.	

\bibitem{McG} D. McGavran. \emph{Adjacent connected sums and torus actions}, Trans. Amer. Math. Soc., {\bf{251}} (1979), 235--254.

\bibitem{P} Taras E. Panov. \emph{Cohomology of face rings, and torus actions},
in ``Surveys in Contemporary Mathematics''. London Math. Soc.
Lecture Note Series, vol.~\textbf{347}, Cambridge, U.K., 2008, 165--201; arXiv:math.AT/0506526.

\bibitem{Post05} A.~Postnikov. \emph{Permutohedra, associahedra, and beyond}, Int. Math. Res. Not. (2009), no. 6, 1026--1106; arXiv: math.CO/0507163.

\bibitem{PRW06} A.~Postnikov, V.~Reiner, L.~Williams. \emph{Faces of generalized permutohedra},  Documenta Mathematica, {\bf{13}} (2008), 207--273; arXiv:math/0609184 v2.

\bibitem{S}
James D. Stasheff. \emph{Homotopy associativity of H-spaces. I.}, Trans. Amer. Math. Soc.~\textbf{108} (1963), 275--292.

\end{thebibliography}
\end{document}